\documentclass[11pt]{article}
\usepackage[dvips]{epsfig}
\usepackage{amsmath,amssymb,euscript,graphicx,amsfonts,verbatim}
\usepackage{enumerate,color}
\usepackage{graphicx}
\usepackage{hyperref}

\usepackage{graphicx}

\newtheorem{theorem}{Theorem}[section]
\newtheorem{proposition}[theorem]{Proposition}
\newtheorem{lemma}[theorem]{Lemma}
\newtheorem{corollary}[theorem]{Corollary}

\newtheorem{conjecture}[theorem]{Conjecture}

\newtheorem{example}[theorem]{Example}

\newtheorem{remark}[theorem]{Remark}

\newenvironment{proof}{{\noindent \sc Proof. }}{\hfill $\Qed$\\}

\newcommand{\la}{\langle}
\newcommand{\ra}{\rangle}

\newcommand{\Qed}{\rule{2.5mm}{3mm}}

\newcommand{\Cay}{\hbox{{\rm Cay}}}

\newcommand{\ZZ}{\mathbb{Z}}

\newcommand{\B}{{\cal{B}}}

\newcounter{case}

\renewcommand{\thecase}{\arabic{case}}

\newcounter{subcase}

\numberwithin{subcase}{case}

\newenvironment{proofT}{{\noindent \sc Proof of Theorem~\ref{the:main}.}}{\hfill $\Qed$ \\}

\newenvironment{proofT2}{{\noindent \sc Proof of Theorem~\ref{the:main2}.}}{\hfill $\Qed$ \\}

\begin{document}


\begin{center}
{\bf\large ON INTERSECTION DENSITY OF TRANSITIVE GROUPS\\
OF DEGREE A PRODUCT OF TWO ODD PRIMES} \\ [+4ex]
Ademir Hujdurovi\'c{\small$^{a,b,}$}\footnotemark,
Klavdija Kutnar{\small$^{a,b,}$}\footnotemark$^{,*}$,
\addtocounter{footnote}{0}
Bojan Kuzma{\small$^{a, b, c,}$,}\footnotemark 
\addtocounter{footnote}{0}\\
Dragan Maru\v si\v c{\small$^{a, b, c,}$}\footnotemark  
\v Stefko Miklavi\v c{\small$^{a, b, c,}$,}\footnotemark 
\ and
Marko Orel{\small$^{a, b, c,}$}\footnotemark 
\\ [+2ex]
{\it \small
$^a$University of Primorska, UP IAM, Muzejski trg 2, 6000 Koper, Slovenia\\
$^b$University of Primorska, UP FAMNIT, Glagolja\v ska 8, 6000 Koper, Slovenia\\
$^c$IMFM, Jadranska 19, 1000 Ljubljana, Slovenia}
\end{center}

\addtocounter{footnote}{-5}
\footnotetext{The work of Ademir Hujdurovi\'c  is supported in part by the Slovenian Research Agency (research program P1-0404 and research projects
N1-0062, J1-9110, N1-0102, J1-1691, J1-1694, J1-1695, N1-0140,
N1-0159, J1-2451 and N1-0208).}
\addtocounter{footnote}{1}
\footnotetext{The work of Klavdija Kutnar  is supported in part by the Slovenian Research Agency (research program P1-0285 and research projects
N1-0062, J1-9110, J1-9186, J1-1695, J1-1715, N1-0140, J1-2451,
 J1-2481 and N1-0209).}
\addtocounter{footnote}{1}
\footnotetext{The work of Bojan Kuzma  is supported in part by the Slovenian Research Agency (research program P1-0285 and research project
N1-0210).}
\addtocounter{footnote}{1}
\footnotetext{The work of Dragan Maru\v si\v c is supported in part by the Slovenian Research Agency (I0-0035, research program P1-0285
and research projects N1-0062,  J1-9108,
J1-1694, J1-1695, N1-0140 and J1-2451).}
\addtocounter{footnote}{1}
\footnotetext{
The work of \v Stefko Miklavi\v c is supported in part by the Slovenian Research Agency (research program P1-0285
and research projects N1-0062,  J1-9110,
J1-1695, N1-0140, N1-0159 and J1-2451).}
\addtocounter{footnote}{1}
\footnotetext{
The work of Marko Orel is supported in part by the Slovenian Research Agency (research program P1-0285
and research projects N1-0140, N1-0208 and N1-0210).

~*Corresponding author e-mail:~klavdija.kutnar@upr.si}

\begin{abstract}
Two elements $g$ and $h$ of a  permutation group $G$  acting on a set $V$ are said to be {\em intersecting} if $g(v) = h(v)$ for some $v \in V$.
More generally, a subset ${\cal F}$ of $G$ is an {\em intersecting set}
if every pair of elements of ${\cal F}$ is intersecting.
The {\em intersection density} $\rho(G)$ of a transitive permutation
group $G$ is the maximum value of the quotient $|{\cal F}|/|G_v|$ 
where $G_v$ is a stabilizer
of $v\in V$ and ${\cal F}$ runs over all intersecting sets in $G$.  
Intersection densities of transitive groups of degree $pq$,
where $p>q$ are odd primes, is considered. In particular,  
the conjecture that the intersection density 
of every such group is equal to $1$
(posed in [{\em J.~Combin. Theory, Ser. A} {\bf 180} (2021), 105390])
is disproved by constructing a family of imprimitive permutation groups 
of degree $pq$ (with blocks of size $q$), 
where $p=(q^k-1)/(q-1)$, whose  intersection
density is equal to $q$. The construction depends heavily on certain
equidistant  cyclic codes  $[p,k]_q$ over 
the field $\mathbb{F}_q$ whose codewords have
Hamming weight strictly smaller than $p$.
\end{abstract}

\begin{quotation}
\noindent {\em Keywords:}
intersection density, transitive permutation group,  cyclic code.
\end{quotation}

\begin{quotation}
\noindent
{\em Math. Subj. Class.:} 05C25, 20B25.
\end{quotation}


\section{Introductory remarks}
\label{sec:intro}
\noindent

Throughout  this paper $p$ and $q$ will always denote   prime numbers
with $p>q$. 

Let $G\le \mathrm{Sym}(V)$ be a permutation group acting on a set $V$,
where $\mathrm{Sym}(V)$ denotes the full symmetric group on $V$.
Two elements $g,h\in G$
are said to be {\em intersecting} if $g(v) = h(v)$ for some $v \in V$.
Furthermore, a subset ${\cal F}$ of $G$ is an {\em intersecting set}
if every pair of elements of ${\cal F}$ is intersecting.
The {\em intersection density} $\rho({\cal F})$ of the intersecting set ${\cal F}$
is defined to be the quotient
$$
\rho({\cal F})=\frac{|{\cal F}|}{\max_{v\in V}|G_v|},
$$
where $G_v$ is the point stabilizer of $v\in V$,
and the {\em intersection density} $\rho(G)$ (see \cite{LSP})
of a group $G$,
is the maximum value of
$\rho({\cal F})$ where ${\cal F}$ runs over all intersecting sets in
$G$, that is,
$$
\rho(G)=\max\{\rho({\cal F})\colon {\cal F}\subseteq G, {\cal F} \textrm{ is intersecting}\} = \frac{\max\{|{\cal F}| \colon {\cal F}\subseteq G\textrm{ is intersecting}\}}{\max_{v\in V}|G_v|}.
$$
Observe that every coset
$gG_v$, $v\in V$ and $g\in G$,
is an intersecting set, referred to as
a {\em canonical} intersecting set.
Clearly, in view of the above,
$\rho(G)\ge 1$. In particular,
for a transitive group $G$ it follows that
$\rho(G)= 1$
if and only if
the maximum cardinality of an intersecting set is $|G|/|V|$.
Following \cite{MRS21} we define ${\cal I}_n$ to be the set
of all intersection densities of transitive permutation groups
of degree $n$:
$$
{\cal I}_n=\{\rho(G) \colon G\textrm{ transitive of degree } n\},
$$
and we let $I(n)$ be the maximum value of ${\cal I}_n$.
%
The {\em derangement graph} $\Gamma_G=\Cay(G,{\cal D})$
is a Cayley graph of $G$ with
the edge set consisting of all pairs $(g,h)\in G\times G$ such that
$gh^{-1}\in {\cal D}$, where ${\cal D}$ is the set of all
fixed-point-free elements (i.e. derangements) of $G$.

The following conjecture was posed in \cite{MRS21}.

\begin{conjecture}
\label{conj:spiga} {\rm\cite[Conjecture~6.6]{MRS21}}
Let $G$ be a transitive permutation group of degree $n$. Then
the following hold.
\begin{enumerate}[(i)]
\itemsep=0pt
\item If $n$ is even, but not a power of $2$,
then there is a transitive group $H$ of degree $n$ with
$\Gamma_H$ a complete multipartite graph with $n/2$ parts.
\item If $n$ is a prime power, then $I(n)=1$.
\item If $n=pq$ where $p$ and $q$ are odd primes, then $I(n)=1$.
\item If $n=2p$ where $p$ is a prime, then $I(n)=2$.
\end{enumerate}
\end{conjecture}

Conjectures~\ref{conj:spiga}(ii) and (iv) were settled, respectively,
in \cite{HKMM21} and \cite{R21}.
Furthermore in \cite{HKMM21} it was  shown that
$\mathcal{I}_{2p}=\{1,2\}$ for every odd prime $p$ and a complete
characterization of   groups of degree $2p$  with
intersection density $2$ was also given there.
In Proposition~\ref{pro:blocks-p} 
we show that Conjecture~\ref{conj:spiga}(iii)
is true for  transitive groups
of degree $pq$, $p>q$ odd primes, which  contain a transitive
subgroup with blocks of size $p$. 
Moreover, as a main result of this paper we give a construction
of transitive groups of degree $pq$ with blocks of size $q$ and
intersection density $q$ (see Theorem~\ref{the:main} below). 
Consequently   Conjecture~\ref{conj:spiga}(iii) is not true.

The above construction  relies heavily on certain cyclic codes 
to which  imprimitive permutation groups are associated in the following way.
(See Section~\ref{sec:pre} for basic properties of cyclic codes
needed in this construction.) 
Let $C$ be a cyclic code of length $m$ over $\mathbb{F}_q$ and let $V=\ZZ_q\times\ZZ_m$. Let 
$\alpha \in Sym(V)$ act according to the rule 
$$
\alpha\colon (i,j) \mapsto (i,j+1) \textrm{ for $i\in\ZZ_q$ and $j\in\ZZ_m$}.
$$
To each $\mathbf{c}=(c_0,c_1,\ldots,c_{m-1})\in C$ a
 permutation  $\beta_{\mathbf{c}}$ acting according to the rule
$$
\beta_{\mathbf{c}} \colon (i,j) \mapsto (i+c_j,j) \textrm{ for $i\in\ZZ_q$ and $j\in\ZZ_m$}
$$
is assigned.
Finally, we let $G(C)\leq Sym(V)$ be the permutation group 
generated by $\{\alpha\}\cup \{\beta_{\mathbf{c}}\mid \mathbf{c}\in C \}$. 
Clearly, $G(C)$ is an imprimitive permutation group with $m$ blocks
$\{(i,j)\colon i\in \ZZ_q\}$ of size $q$.

The following two theorems are the main results of this paper.

\begin{theorem}
\label{the:main2}
Let $q$ be a prime and $m$ a positive integer. 
Let $C$ be a nonzero 
cyclic code over $\mathbb{F}_q$ of length $m$
such that no codeword has maximal Hamming weight $m$.
Then the corresponding permutation group $G(C)$ of $\ZZ_q\times\ZZ_m$ has intersection density equal to $q$.
\end{theorem}

\begin{theorem}
\label{the:main}
Let $p$ and $q$ be odd primes such that $p=\frac{q^k-1}{q-1}$ for some positive integer $k$. Then there exists an imprimitive group of degree $pq$ with
blocks of size $q$ whose  intersection density equals $q$.
\end{theorem}

In Section~\ref{sec:pq} the structure of transitive permutation groups
of degree $pq$ is described, 
in Section~\ref{sec:pre} basic properties of cyclic codes are presented,
and in Section~\ref{sec:xx}
the proofs of Theorems~\ref{the:main2} and~\ref{the:main} are given.


\section{Hierarchy of transitive groups of degree $pq$}
\label{sec:pq}
\noindent

Let $G$ be a
transitive permutation group $G$ acting on a set $V$.
A partition $\B$ of $V$ is called $G$-{\em invariant}
if the elements of $G$ permute the parts, the so called
{\em blocks} of $\B$, setwise.
If the trivial partitions $\{V\}$ and $\{\{v\}: v \in V\}$
are the only
$G$-invariant partitions of $V$, then $G$ is {\em primitive},
and is {\em imprimitive} otherwise.
In the latter case a corresponding nontrivial
$G$-invariant partition will be referred to as a
{\em complete imprimitivity block system} of $G$.
We say that $G$ is {\em doubly transitive} if given any two ordered pairs
$(u,v)$ and $(u',v')$ of  elements $u,v,u',v'\in V$, such that $u\ne v$ and $u'\ne v'$,
there exists an element  $g\in G$ such that $g(u,v)=(u',v')$.
Note that a doubly transitive group is primitive.
A primitive group which is not doubly transitive is called {\em simply primitive}.

A transitive group of degree $pq$ falls into one of the following
three classes: it either has blocks of size $p$ or it has blocks
of size $q$ or it is a primitive group. (For a detailed description
of these groups see \cite{MS2,MS5,PWX,PX}.)
In the latter case the group is either doubly transitive in which
case the intersection density is known to be equal to $1$
(see \cite[Lemma~2.1(3)]{MRS21}) or it is simply primitive.
Using the fact  that $\rho(G)\le \rho(H)$ for transitive groups
$H\le G$ (see \cite[Lemma~6.5]{MRS21})
the list of all simply primitive groups is further
reduced to a shorter list of simply primitive groups containing
no imprimitive subgroups (see \cite{DKM21,MS5}).
For each group on this list the corresponding
intersection density will have
to be computed.
Coming back to imprimitive groups, the first of the above three
classes is the easiest to deal with
when considering intersection density. The following holds.

\begin{proposition}
\label{pro:blocks-p}
Let $G$ be a transitive group of degree $pq$, $p>q$ primes,  containing
an imprimitive subgroup $H$ with $q$ blocks of size $p$.
Then $\rho(G)=1$.
\end{proposition}

\begin{proof}
It may be seen that $H$ contains  a derangement of order $p$,
in fact a semiregular element $\alpha$ of order $p$ such that
the set of orbits of $\la\alpha\ra$ forms an $H$-invariant partition $\B$,
see \cite{DM81}. Let $\bar{H}$ be the permutation group
induced by the action of $H$ on $\B$.
Then, by \cite[Lemma~3.1]{HKMM21}, $\rho(H)\leq \rho(\bar{H})$.
Since $\bar{H}$ is a transitive group  of prime degree we have
  $\rho(\bar{H})=1$ by \cite[Theorem~1.4]{HKMM21},
and so $\rho(H)=1$.
Moreover, by \cite[Lemma~6.5]{MRS21},
$\rho(G)\le \rho(H)$,  and the result follows.
\end{proof}

This leaves us with   imprimitive groups having blocks of size $q$ and
containing no transitive subgroups having blocks of size $p$.
The case $q=2$ was settled in \cite{HKMM21,R21}.
As mentioned in Section~\ref{sec:intro} 
we will construct in Section~\ref{sec:xx}
a family of imprimitive groups with blocks of size $q$
having intersection density equal to $q$, thus
disproving Conjecture~\ref{conj:spiga}(iii).
In Example~\ref{ex:33} below we give the smallest counter-example to this
conjecture, a group  of degree $3\cdot 11$ with intersection density $3$.
This group, however, is  not part of the family 
of groups from Theorem~\ref{the:main}
where the smallest group occurs for $q=3$
and $p=13$.
In summary, in order to obtain
a complete characterization of intersection densities
of transitive groups of degree $pq$, simply primitive groups
from the above mentioned list and imprimitive groups
with blocks of size $q$ will have to be addressed.

\begin{example}\label{ex:33}
{\rm
Let $a,b_0\in Sym(\ZZ_{33})$ be defined with
$a(i)=i+3 \pmod{33}$ and $b_0=(0\,1\,2)$. For $k\in \{1,\ldots,10\}$ let $b_k=b_0^{a^k}=(3k\,\,3k+1\,\, 3k+2)$ and let
$b=b_0\cdot b_2 \cdot b_3^2 \cdot b_4^2 \cdot b_5^2 \cdot b_6$. Define
 $G=\langle a,b \rangle \leq Sym(\ZZ_{33})$. It can be verified (with MAGMA, for example) that $G$ is a transitive group of order $3^5 \cdot 11$ admitting blocks  $\{3k,3k+1,3k+2\}$, $k\in\ZZ_{11}$,
of size 3. The kernel $K$ of the action of $G$ on these blocks 
 is an elementary abelian group of order $3^5$, and contains 
no non-identity semiregular element. Hence $K$ is 
an intersecting set of size $3^5=3\cdot |G_v|$. 
This shows that $\rho(G)\geq 3$. 
Since $G$ admits a semiregular subgroup $\langle a \rangle$ with three orbits, by \cite[Proposition 2.6]{HKMM21} it follows that $\rho(G)=3$.

}
\end{example}


\section{Cyclic codes}
\label{sec:pre}
\noindent

Let $m$ be a positive integer, $r$ a power of a prime, and $\mathbb{F}_r$ the finite field with $r$ elements. The polynomial $x^m-1\in \mathbb{F}_r[x]$ has no repeated factors (which are irreducible over $\mathbb{F}_r$) if and only if $r$ and $m$ are relatively prime, i.e. $gcd(r,m)=1$ 
(see \cite[Exercise~201]{pless}), which we assume in this section.

Let $\mathbb{F}_r^m$ be the $m$-dimensional vector space over $\mathbb{F}_r$  formed by all row vectors $(c_0,c_1,\ldots,c_{m-1})$ with entries in $\mathbb{F}_r$. Let $C$ be a linear $[m,k]_r$ code, that is, a $k$-dimensional 
vector subspace in $\mathbb{F}_r^m$. A linear code $C$ is \emph{cyclic} if $(c_0,c_1,\ldots,c_{m-1})\in C$ implies $(c_{m-1},c_0,\ldots,c_{m-2})\in C$. The vector space $\mathbb{F}_r^m$ can be identified with the principal ideal domain $\mathbb{F}_r[x]/(x^m-1)$. Under this identification,  cyclic codes correspond 
exactly to the ideals in $\mathbb{F}_r[x]/(x^m-1)$ (see \cite[Theorem~9.36]{LN} or \cite[Theorem~4.2.1]{pless}). 
The \emph{generating polynomial} $g(x)$ of a nonzero cyclic code $C$ is
the unique monic polynomial of the lowest degree in $C$. In this case
$C=\langle g(x)\rangle:=\{a(x)g(x) \colon a(x)\in \mathbb{F}_r[x])\}$,
where the multiplication is done modulo $x^m-1$. Moreover, $g(x)$
divides the polynomial $x^m-1$ in $\mathbb{F}_r[x]$, and the dimension
of the cyclic code $C$ equals $k=m-\deg g(x)$.
The polynomial $h(x)=(x^m-1)/g(x)$ 
is the \emph{parity-check polynomial} of  $C$. If $g(x)=\sum_{i=0}^{m-k}g_i x^i$ where $g_i\in \mathbb{F}_r$, then $C$, viewed in $\mathbb{F}_r^m$, is spanned by $k$ vectors 
$$
(g_0,g_1,\ldots,g_{m-k},0,\ldots,0), (0,g_0,g_1,\ldots,g_{m-k},0,\ldots,0), \ldots,  (0,\ldots,0,g_0,g_1,\ldots,g_{m-k}).
$$
Let
\begin{equation}\label{eq1Marko}
\Phi_m(x)=\prod_{d|m}(x^d-1)^{\mu(m/d)}
\end{equation}
be the \emph{$m$-th cyclotomic polynomial}. Here $\mu$ is the M\"{o}bius function 
$$\mu(t)=\left\{\begin{array}{lll}1 &\textrm{if}\ t=1,\\
0& \textrm{if a square of some prime divides}\ t,\\
(-1)^s& \textrm{if}\ t\ \textrm{is a product of}\ s\ \textrm{distinct primes}.\end{array}\right.$$
Then, by \cite[Equation~9.20]{wan}, we have 
\begin{equation}\label{eq2Marko}
x^m-1=\prod_{d|m}\Phi_d(x).
\end{equation}
By \cite[Theorem~9.14]{wan}, the function \eqref{eq1Marko} is indeed a polynomial with integer coefficients. Hence, cyclotomic polynomials can be understood also as elements in $\mathbb{F}_r[x]$. Since $gcd(r,m)=1$,   \cite[Theorem~9.16]{wan} implies that $\Phi_m(x)$ is a product of $\phi(m)/k$ distinct monic polynomials in $\mathbb{F}_r[x]$ that are irreducible over $\mathbb{F}_r$ and of degree $k$, which is the least positive integer such that 
$r^k=1\, (\textrm{mod}~m).$
Here, $\phi$ is the Euler function. If $h(x)\in \mathbb{F}_r[x]$ is one of the 
 irreducible factors of $\Phi_m(x)$  over $\mathbb{F}_r$,
then, by \eqref{eq2Marko}, 
it divides   the polynomial $x^m-1\in \mathbb{F}$. 
Hence, $h(x)$ is the parity-check polynomial of the cyclic code $\langle g(x)\rangle$, where $g(x)=(x^m-1)/h(x)$. The following result is proved in \cite[Equation~2.10]{mceliece} (see also~\cite{niederreiter77}).

\begin{lemma}[\cite{mceliece}]\label{lemmaMarko1}
Let $gcd(r,m)=1$ and let $h(x)\in \mathbb{F}_r[x]$ be a monic factor in $\Phi_m(x)$, which is irreducible over $\mathbb{F}_r$ and of degree $k$. If ${\bf c}$ is any nonzero codeword in the cyclic $[m,k]_r$ code $C$ with the parity-check polynomial $h(x)$, then the number $Z({\bf c})$ of zero entries in ${\bf c}$ satisfies
\begin{equation}\label{eq3Marko}
\left|Z({\bf c})-\frac{(r^{k-1}-1)m}{r^k-1}\right|\leq \left(1-\frac{1}{r}\right)\left(\frac{gcd(m,r-1)}{r-1}-\frac{m}{r^k-1}\right)r^{k/2}.
\end{equation}
\end{lemma}

Recall that the value $w({\bf c}):=m-Z({\bf c})$ is the \emph{Hamming weight} of a codeword ${\bf c}$ in a $[m,k]_r$ code. If a code in Lemma~\ref{lemmaMarko1} satisfies
\begin{equation}\label{eq4Marko}
\frac{r^k-1}{r-1}=\frac{m}{gcd(m,r-1)},
\end{equation}
then all its nonzero codewords have constant weight equal to
\begin{equation}\label{eq6Marko}
m-\frac{(r^{k-1}-1)m}{r^k-1}=m-(1+r+r^2+\cdots+r^{k-2})gcd(m,r-1).
\end{equation}
In this case, the linearity of the code implies that the \emph{Hamming distance} $d({\bf c}_1,{\bf c}_2):=w({\bf c}_1-{\bf c}_2)$ attains constant value~\eqref{eq6Marko} for all distinct codewords ${\bf c}_1,{\bf c}_2\in C$, and the code is referred to as \emph{equidistant}.


\section{A family of groups with intersection density $q$}
\label{sec:xx}
\noindent


We start this section by proving Theorem~\ref{the:main2}.
Recall from the introductory section
that to every cyclic code $C$ of length $m$ over $\mathbb{F}_q$
we can associate 
an imprimitive permutation group $G(C)$
acting on $\ZZ_q\times\ZZ_m$. 

\bigskip

\begin{proofT2}
Let $C$ be a nonzero cyclic code of length $m$ over $\mathbb{F}_q$ such that no codeword has maximal Hamming weight $m$, and let $G(C)$ be the permutation group associated with $C$. Let $K$ be the subgroup of $G(C)$ generated by $\{\beta_{\mathbf{c}}\mid \mathbf{c}\in C \}$. Observe that $\beta:C\to K$ defined by $\beta(\mathbf{c})= \beta_{\mathbf{c}}$ is an isomorphism between the additive group of the code $C$ and the group $K$. Therefore, $K$ is an elementary abelian group of order $q^k$, where $k$ is the dimension of $C$. Observe that $K$ is normalized by $\alpha$, hence $G(C)\cong K\rtimes \langle \alpha \rangle$. It follows that $|G(C)|=mq^k$. 

Since $C$ is a nonzero cyclic code, it follows that for each $j\in \ZZ_m$ there exists $\mathbf{c} \in C$ with $c_j\neq 0$. Considering the action of $\beta_{\mathbf{c}}$  it follows that $K$ acts transitively on each of the sets $\ZZ_q\times \{j\}$, for each $j\in \ZZ_m$, and using the fact that $\alpha$ permutes the sets $\ZZ_q\times \{j\}$ it follows that $G(C)$ acts transitively on $\ZZ_q\times \ZZ_m$. By the orbit-stabilizer theorem, it follows that the order of a point stabilizer in $G(C)$ is $\frac{mq^k}{mq}=q^{k-1}$. 

Let $\mathbf{c}\in C$ and $j\in \ZZ_m$ such that $c_j=0$. 
Observe that $ \beta_{\mathbf{c}}$ fixes each element of the set $\ZZ_q\times \{j\}$. The assumption that each codeword of $C$ has a zero entry implies that each element of $K$ has a fixed point. Since $K$ is a subgroup of $G(C)$ it follows that $K$ is an intersecting set of $G(C)$ of size $q ^k$. We conclude that $\rho(G)\geq q$. 

On the other hand, since $G(C)$ admits a semiregular subgroup $\langle \alpha \rangle$ with $q$ orbits, by \cite[Proposition 2.6]{HKMM21} it follows that $\rho(G(C))\leq q$, and so  $\rho(G(C))=q$.
\end{proofT2}

Theorem~\ref{the:main2} gives us a method of constructing transitive 
permutation groups with intersection density equal to $q$. However, the 
construction of cyclic codes of prime length $p$ over the field $\mathbb{F}_q$ 
with no codeword having the maximal Hamming weight $p$ is an intriguing 
problem, which we address in the remainder of this section. The smallest 
example is an $[11,5]_3$ cyclic code, and the corresponding permutation group 
$G(C)$ is given in Example~\ref{ex:33}.
The following lemma will be needed.

\begin{lemma}\label{lemmaMarko2}
Let $r$ be a power of a prime, and let 
$k$ and $m$ be positive integers 
satisfying~\eqref{eq4Marko}. Then $gcd(m,r)=1$ and $k$ is the smallest positive integer such that $r^k\equiv1\, (\textrm{mod}~m)$.
\end{lemma}

\begin{proof}
 From~\eqref{eq4Marko} we deduce that
\begin{equation}\label{eq5Marko}
r^k-1=\frac{r-1}{gcd(m,r-1)}m.
\end{equation}
Let $r$ be a power of a prime $r_0$.
Suppose that $gcd(m,r)>1$. Then $m\equiv0\, (\textrm{mod}~r_0)$. Since $\frac{r-1}{gcd(m,r-1)}$ is an integer, we deduce that the right-hand side of~\eqref{eq5Marko} vanishes modulo $r_0$, while the left-hand side equals $-1$, a contradiction. Hence, $gcd(m,r)=1$.

From~\eqref{eq5Marko} it is clear that $r^k\equiv 1\, (\textrm{mod}~m)$. Let $i$ be any positive integer such that $r^i\equiv 1\, (\textrm{mod}~m)$. Then $r^i-1=a m$ and $r-1=b\cdot gcd(m,r-1)$ for some integers $1\leq a$ and $b\leq r-1$. Hence we deduce from~\eqref{eq5Marko} that
$$m\frac{r^k-1}{r^i-1}\leq am\frac{r^k-1}{r^i-1}=\frac{r-1}{gcd(m,r-1)}m=bm\leq m(r-1),$$
which yields $r^i-1\geq 1+r+r^2+\cdots+ r^{k-1}$. Hence, $i\geq k$ as claimed.
\end{proof}

The next result is deduced immediately from Lemmas~\ref{lemmaMarko1} and~\ref{lemmaMarko2}.
\begin{corollary}\label{corMarko1}
Let $r$ be a power of a prime, and let 
$k$ and $m$ be positive integers 
satisfying~\eqref{eq4Marko}. 
If $h(x)\in \mathbb{F}_r[x]$ is a monic factor in $\Phi_m(x)$, which is irreducible over $\mathbb{F}_r$, then its degree is $k$, and the cyclic linear $[m,k]_r$ code with the parity-check polynomial $h(x)$ is equidistant and all its nonzero codewords have weight
$$m\cdot \frac{r^{k}-r^{k-1}}{r^k-1}.$$
\end{corollary}

Cyclic equidistant codes were characterized in~\cite{clark}. Actually, they are just a bit more general then the codes from Corollary~\ref{corMarko1}. 
(If you repeat a cyclic equidistant code you obtain a cyclic equidistant code.) 
For $m$  a prime, 
Corollary~\ref{corMarko1} provides all cyclic equidistant codes. For a characterization of all (not necessarily cyclic) linear equidistant codes see~\cite{bonisoli}.

In what follows we restrict the conditions on parameters $m,r,k$ in Corollary~\ref{corMarko1}.  An (odd) prime $p=m$ of the form $$p=\frac{r^k-1}{r-1},$$ where $r$ is a power of some prime, is said to be
\emph{projective}~\cite{JZ}. 
In this case it may be seen that $k$ is necessarily
 a prime. Note that a  projective prime with $k=2$ 
is necessarily a Fermat prime, and a projective prime 
with $r=2$ is a Mersenne prime (see~\cite{JZ}).

\begin{corollary}\label{corMarko2}
Let $r$ be a power of a prime, and let $p=\frac{r^k-1}{r-1}$ be a projective prime.
If $h(x)\in \mathbb{F}_r[x]$ is an  irreducible
monic factor in $\Phi_p(x)$  
 over $\mathbb{F}_r$, then its degree is $k$ 
and the cyclic linear $[p,k]_r$ code with the 
parity-check polynomial $h(x)$ is equidistant 
and all its nonzero codewords have $p-r^{k-1}>0$ zero entries.
\end{corollary}
\begin{proof}
Since $\frac{r^k-1}{r-1}=p=\frac{p}{gcd(p,r-1)}$ the claim follows from Corollary~\ref{corMarko1}.
\end{proof}
\begin{example}
The Mersenne prime $$31=\frac{2^{5}-1}{2-1}=\frac{5^3-1}{5-1}$$ induces equidistant cyclic codes of the form $[31,5]_2$ and $[31,3]_5$, where all nonzero codewords have $15$ and $6$ zero entries, respectively. 
\end{example}

We state a special case of Corollary~\ref{corMarko2},
addressing the main goal of this paper, separately.

\begin{corollary}\label{corMarko4}
Let $p$ and $q$ be odd primes such that $p=\frac{q^k-1}{q-1}$ for some positive integer $k$. If $h(x)\in \mathbb{F}_q[x]$ 
is an irreducible monic factor in $\Phi_p(x)$ over $\mathbb{F}_q$, 
then its degree is $k$ and the cyclic linear $[p,k]_q$ code with the 
parity-check polynomial $h(x)$ is equidistant and all its nonzero codewords have $p-q^{k-1}>0$ zero entries.
\end{corollary}

Using Corollary~\ref{corMarko4} the proof of Theorem~\ref{the:main}
is now straightforward.

\bigskip
 
\begin{proofT}
Let $C$ be a $[p,k]_q$ cyclic code constructed in Corollary~\ref{corMarko4}. Since $C$ satisfies the assumptions of Theorem~\ref{the:main2}, it follows that the permutation group $G(C)$ is a transitive permutation group of degree $pq$ with the intersection density $q$. 
\end{proofT}

Projective primes are related to the classification of finite simple groups. Based 
on heuristic arguments and computational evidence Jones and Zvonkin recently 
stated an interesting conjecture claiming that there are infinitely many projective 
primes~\cite[Conjecture~1.3]{JZ}. They made an even stronger conjecture, 
claiming that for any fixed prime $k\geq 3$, there are infinitely 
many primes $q$ such that $$\frac{q^k-1}{q-1}$$
is a prime~\cite[Conjecture~6.5]{JZ}. If this conjecture is true, then there are 
infinitely many pairs $p,q$ of primes that satisfy the assumption in 
Theorem~\ref{the:main}. In any case, the computational 
evidence in~\cite[Table~3]{JZ} implies that there are plenty of such pairs.

\begin{remark}\label{remark2:marko}
Observe that if in Lemma~\ref{lemmaMarko1} we have $$\frac{(r^{k-1}-1)m}{r^k-1}-\left(1-\frac{1}{r}\right)\left(\frac{gcd(m,r-1)}{r-1}-\frac{m}{r^k-1}\right)r^{k/2}>0$$
or equivalently
$$m>gcd(m,r-1)\cdot (r^{k/2}+1)\cdot \frac{r^{k/2-1}}{r^{k/2-1}+1},$$
then $Z({\bf c})>0$. In particular, this is true whenever $m\geq gcd(m,r-1)\cdot (r^{k/2}+1)$, which simplifies into $m\geq r^{k/2}+1$ whenever $m$ is a prime. 

So, let  $p$ and $q$ be odd primes, let $k$ be the smallest positive integer such that $q^k=1\, (\textrm{mod}~p)$, and assume that $p\geq q^{k/2}+1$. 
Then by the argument in the previous paragraph
we again have that each  codeword has some zero entries, leading to a construction of a transitive permutation group of degree
$pq$ and intersection density $q$.

For example, the parameters $p=757$, $q=3$, $k=9$ have these properties. Namely, $q^k-1=26\cdot 757$, $p\geq q^{k/2}+1\doteq 141.3$, and it can we verified that $q^i\neq 1 \, (\textrm{mod}~p)$ for $1\leq i\leq 8$.

\end{remark}



 \end{document}